\documentclass[11pt]{article}
\usepackage{amssymb,amsmath,amsfonts,amsthm,bbm}
\usepackage[english]{babel}
\usepackage[ansinew]{inputenc}
\usepackage{graphics,epsfig}

\theoremstyle{plain}
\newtheorem{theorem}{Theorem}[section]
\newtheorem{lemma}[theorem]{Lemma}
\newtheorem{proposition}[theorem]{Proposition}
\newtheorem{corollary}[theorem]{Corollary}
\newtheorem{remark}[theorem]{Remark}

\newcommand{\argmax}{\operatornamewithlimits{argmax}}

\begin{document}

\title{Timing in the Presence of Directional Predictability:\\ Optimal Stopping of Skew Brownian Motion}

\author{Luis H. R. Alvarez E.\thanks{Department of Accounting and Finance, Turku School of Economics,
FIN-20014 University of Turku, Finland, E-mail: luis.alvarez@tse.fi}\quad
Paavo Salminen\thanks{Faculty of Science and Engineering, Åbo Akademi University, FIN-20500 Åbo,
Finland, E-mail: paavo.salminen@abo.fi}}
\maketitle
\abstract{We investigate a class of optimal stopping problems arising in, for example, studies considering the timing of an irreversible investment when the underlying follows a skew Brownian motion.
Our results indicate that the local directional predictability modeled by the presence of a skew point for the underlying has a nontrivial and somewhat surprising impact on the timing incentives of the decision maker. We prove that waiting is always optimal at the skew point for a large class of exercise payoffs. An interesting consequence of this finding, which is in sharp contrast with studies relying on
ordinary Brownian motion, is that the exercise region for the problem can become unconnected even when the payoff is linear. We also establish that higher skewness increases the incentives to wait and postpones the optimal timing of an investment opportunity. Our general results are explicitly illustrated for a piecewise linear payoff.}\\



\thispagestyle{empty} \clearpage \setcounter{page}{1}
\section{Introduction}

Standard Brownian motion constitutes without a doubt the most commonly utilized model for the factor dynamics driving the underlying stochasticity in financial models.
Its analytical tractability and computational facility makes it a compelling model with many desirable properties ranging from the independence of its increments to
the Gaussianity of its probability distribution. Unfortunately, for many financial return variables the presence of autocorrelation of the driving dynamics and/or skewness of the probability distributions constitutes a rule rather than an exception. It is clear that in such a case relying on a simple Gaussian structure may result in wrong conclusions concerning both the valuation and the timing of investment
opportunities.

In contrast with the standard Gaussian framework, relatively recent empirical research indicates that even though the exact value of an asset is unpredictable, the direction towards which the asset value is expected to develop may be predictable to some extent (see, for example, \cite{AnGe}, \cite{AnGo}, \cite{BeGe1}, \cite{BeGe2}, \cite{Ch}, \cite{ChDi}, \cite{Chetal}, \cite{Ny}, \cite{RySh}, and \cite{Sk}). More precisely, expressing the return of an asset as the product of its sign and its absolute value and investigating the behavior of these factors separately indicates that the sign variable capturing the directional behavior of the return can be forecasted correctly with an accuracy ranging from 52\% to even 60\% (for a recent survey of studies focusing on directional predictability, see \cite{Ham}). This empirical observation has not went completely unnoticed in theoretical finance studies and it has resulted into the introduction and the analysis of driving dynamics possessing at least some of the skewness and the local (in space) predictive properties encountered in financial data. One of the proposed modeling approaches is based on skew Brownian motion and skew diffusion processes in general (cf. \cite{CoSa}, \cite{DeScGoSc}, \cite{DeGoSc,DeGoSc2}, and \cite{Ro}). Basically, a skew Brownian motion behaves like an ordinary Brownian motion outside the origin (see, for example, \cite{Appuetal}, \cite{Appuetal2}, \cite{Ba}, \cite{BS}, \cite{BuCh}, \cite{HaSh}, \cite{IM}, \cite{Le}, \cite{LeMoTo}, \cite{Ou}, \cite{Vu}, \cite{Wa}). However, at the origin the process has more tendency to move, say, upwards than downwards resulting in a sense  into a larger number of positive than negative excursions starting from the origin. In that way it offers a mathematical model for local directional predictability of the driving random factor and, consequently, to an asymmetric and skewed probability distribution of the underlying random dynamics.

In this paper we investigate how the singularity generated by the skewness of the underlying driving diffusion affect optimal stopping policies within an infinite horizon setting. Our approach for solving the considered optimal stopping problem is based on the scrutinized analysis of the superharmonic functions (see, for example, \cite{Al}, \cite{BeLe1}, \cite{BeLe}, \cite{ChIr}, \cite{CrMo}, \cite{DaKa}, \cite{EkVi}, \cite{Sa2}, and \cite{SaTa} and references therein). In particular, we use the Martin representation theory of superharmonic functions (cf. \cite{ChIr} and \cite{salminen}).
We demonstrate that positive skewness increases the incentives to wait at the singularity so radically that the skew point is always included in the continuation region provided that the exercise payoff is increasing at the skew point. This observation is in sharp contrast with results based on standard Brownian motion and illustrates how even relatively small local predictability of the underlying diffusion generates incentives to wait and, in that way, postpone the optimal stopping of the underlying process. An interesting and to some extent surprising implication of this observation is that the optimal stopping policy for skew BM can become a three-boundary policy even in the case where the exercise payoff is piecewise linear (call option type). Such configurations cannot appear in models relying on standard BM. We also demonstrate that the sign of the dependence of the value of the optimal policy and the skewness of the underlying diffusion is positive. Consequently, higher skewness increases the value of the optimal policy and expands the continuation region. An interesting implication of this observation is that the value of the optimal stopping strategy  for a positively skew BM dominates the corresponding value for standard BM.

The contents of this study are as follows. The basic properties of the underlying dynamics,
i.e., skew Brownian motion, are discussed in Section 2. In Section 3
the considered stopping problem and some key facts are presented. Our main findings on optimal stopping of skew Brownian motion are summarized in Section 4. These results are then numerically illustrated in an explicitly parameterized piecewise linear model in Section 5. Finally, Section 6 concludes our study.

\section{Underlying Dynamics: Skew Brownian Motion}%  and Problem Setting}
Our main objective is to investigate how the potential directional asymmetry of the underlying diffusion affects the optimal exercise strategies and their values.
In order to accomplish this task, we assume that the underlying diffusion process is a {\em skew Brownian motion} (abbreviated from now on as SBM) characterized as the unique strong solution of the SDE (cf. \cite{HaSh})
\begin{align}
X_t = x + W_t + (2\beta-1) l_t^X,\label{skew}
\end{align}
where $x\in \mathbb{R}$ is the initial value of the process,
$\beta\in[0,1]$ is a parameter capturing the skewness of the process,
$\{W_t\}_{t\geq 0}$ is a standard Brownian motion and $\{l_t^X\}_{t\geq 0}$ is the local time at zero of the process $\{X_t\}_{t\geq 0}$ normalized with respect to Lebesgue's measure. As is clear from \eqref{skew}, the process $\{X_t\}_{t\geq 0}$ coincides with standard Brownian motion when $\beta= 1/2$ and with reflected Brownian motion when  $\beta = 0$ or $\beta = 1$. The process $\{X_t\}_{t\geq 0}$ behaves like ordinary Brownian motion outside the skew point $0$ and has for all $t>0$ the property $\mathbb{P}_0[X_t\geq 0]=\beta$ (cf. \cite{BS}, p. 130). Thus, the process has in a sense more tendency  to move up than down from the origin whenever $\beta>1/2$. Moreover, utilizing the known transition probability density  (see, for example, \cite{BS}, p. 130 or \cite{Le}, p. 420)
\begin{align}
\mathbb{P}_x\left[X_t\in dy\right] = \left(\frac{1}{\sqrt{2\pi t}}\;\textrm{e}^{-\frac{(x-y)^2}{2t}}+(2\beta-1)\textrm{sgn}(y)\frac{1}{\sqrt{2\pi t}}\;\textrm{e}^{-\frac{\left(|x|+|y|\right)^2}{2t}}\right)dy,\label{skewdensity}
\end{align}
of SBM yields
\begin{align}
\mathbb{E}_x\left[X_t\right] = x + 2(2\beta-1)\sqrt{t}\;\phi\left(\frac{|x|}{\sqrt{t}}\right)-2(2\beta-1)|x|\; \Phi\left(-\frac{|x|}{\sqrt{t}}\right),\label{skewexpected}
\end{align}
where $\Phi$ is the standard univariate normal distribution function and  $\phi$ is its density. Setting $x=0$ in \eqref{skewexpected} yields
$$
\mathbb{E}_0[X_t] = (2\beta-1)\sqrt{\frac{2t}{\pi}}.
$$
The moment generating function, in turn, reads as
\begin{align*}
\mathbb{E}_x\left[\mathrm{e}^{\lambda X_t}\right] = \textrm{e}^{\lambda x+\frac{1}{2}\lambda^2 t}\left(1+(2\beta-1)\textrm{e}^{-\lambda (|x|+ x)}\Phi\left(\frac{\lambda t-|x|}{\sqrt{t}}\right)
-(2\beta-1)\textrm{e}^{\lambda (|x|-x)}\Phi\left(-\frac{\lambda t+|x|}{\sqrt{t}}\right)\right).
\end{align*}
The scale function and the speed measure of $X$ are given by
$$
S(x) =\begin{cases}
x/\beta, &x\geq0,\\
x/(1-\beta), &x\leq0,
\end{cases}
$$
and
$$
m(dx) =\begin{cases}
2\beta dx, &x>0,\\
2(1-\beta)dx, &x<0,
\end{cases}
$$
respectively. The fact that $S(x)\to\pm\infty$ as $x\to\pm\infty$
implies that $X$ is recurrent. Finally, the increasing and the decreasing fundamental solutions associated with $X$ are (cf. \cite{BS}, p. 130)
\begin{align}\label{increasing}
\psi_r(x)= \textrm{e}^{\theta x}-\left(1-\frac{1}{2\beta}\right)\left(\textrm{e}^{\theta x}-\textrm{e}^{-\theta x}\right)^+=\begin{cases}
 \frac{1}{2\beta}\textrm{e}^{\theta x}+\left(1-\frac{1}{2\beta}\right)\textrm{e}^{-\theta x},&x\geq0,\\
\textrm{e}^{\theta x},&x\leq0,
\end{cases}
\end{align}
and
\begin{align}\label{decreasing}
\varphi_r(x)=  \textrm{e}^{-\theta x} +\frac{2\beta-1}{2(1-\beta)}(\textrm{e}^{-\theta x}-\textrm{e}^{\theta x})^+=\begin{cases}
\textrm{e}^{-\theta x},&x\geq0,\\
 \frac{1}{2(1-\beta)}\left((1-2\beta)\textrm{e}^{\theta x}+\textrm{e}^{-\theta x}\right),&x\leq0,
\end{cases}
\end{align}
respectively, where $\theta=\sqrt{2r}$ is the so-called Wronskian of
the fundamental solutions with respect to the scale function. It is
easily seen that $\psi_r$ and $\varphi_r$ are differentiable with
respect to $S$ everywhere (also at 0), but not in the ordinary sense at
0.

\section{Problem Setting and Some Preliminary Results}

 Our task is to investigate for  SBM $X$ with $\beta>1/2$ how the
 skewness and the resulting local directional predictability of the
 underlying affects the value and optimal exercise policy in the
 optimal stopping problem (OSP):

\noindent
{\sl{Find a stopping time $\tau^\ast$ such that
\begin{align}\label{osp}
V(x):=\sup_{\tau\in \mathcal{T}}\mathbb{E}_x\left[\textrm{e}^{-r\tau}g(X_\tau)\right] = \mathbb{E}_x\left[\textrm{e}^{-r{\tau^\ast}}g(X_{\tau^\ast})\right],
\end{align}
where $r>0$ denotes the prevailing discount rate, $\mathcal{T}$ is the set of all stopping times with respect to the natural filtration generated by $X$, and $g:\mathbb{R}\mapsto \mathbb{R}_+$ is the exercise reward satisfying:
\begin{itemize}
\item[\rm(g1)] $g$ is continuous, non-decreasing, non-negative, and has finite left and right derivatives,
\item[\rm(g2)]  $\lim_{x\rightarrow \infty}g(x)/\psi_r(x)=0$ and $\lim_{x\rightarrow -\infty}g(x)/\psi_r(x)=0$.
\end{itemize}
}}
In \eqref{osp} we use the convention that if $\tau(\omega)=\infty$ then
$$
\textrm{e}^{-r\tau(\omega)}g(X_{\tau(\omega)}(\omega)):=\limsup_{t\rightarrow\infty}\textrm{e}^{-rt}g(X_t(\omega)).
$$

As is known from the literature on optimal stopping $V$ is the smallest $r$-excessive majorant of $g$ (cf. Theorem 1 on p. 124 of \cite{Sh}).
As usual, we call $\Gamma:=\{x:V(x)=g(x)\}$ the stopping region and $C:=\{x:V(x)>g(x)\}$ the continuation region.
Let
\begin{align}
\mathcal{M}:=\argmax_{x\in  \mathbb{R}}\{g(x)/\psi_r(x)\}\label{M}
\end{align}
denote the set of points at which the ratio $g/\psi_r$ is maximized.
We can now prove the following:
\begin{lemma}\label{lemma1}
The value of the optimal policy is finite, i.e. $V(x)<\infty$ for all $x\in \mathbb{R}$, and the stopping region is nonempty, i.e. $\Gamma\neq\emptyset$.
\end{lemma}
\begin{proof}
Assumptions (g1) and (g2) guarantee that the set of maximizers $\mathcal{M}$ is non-empty.
Hence, for all $x\in \mathbb{R}$ it holds that
\begin{align}\label{bound}
\nonumber
V(x)=\sup_{\tau\in
    \mathcal{T}}\mathbb{E}_x\left[\textrm{e}^{-r\tau}\frac{g(X_\tau)}{\psi_r(X_\tau)}\psi_r(X_\tau)\right]
&  \leq  \sup_{y\in \mathbb{R}}\frac{g(y)}{\psi_r(y)}\sup_{\tau\in
    \mathcal{T}}\mathbb{E}_x\left[\textrm{e}^{-r\tau}\psi_r(X_\tau)\right]
\\
&\leq \psi_r(x)\sup_{y\in \mathbb{R}}\frac{g(y)}{\psi_r(y)}.
\end{align}
For the last inequality in \eqref{bound} we use the optional
sampling theorem which is
justified since $\{\textrm{e}^{-rt}\psi_r(X_t)\}_{t\geq 0}$ is a positive supermartingale. This proves that $V(x)<\infty$ for all $x\in \mathbb{R}$. In order to show that $\Gamma\neq\emptyset$ let $x^\ast\in \mathcal{M}$ and utilize \eqref{bound} to obtain
$$
V(x^\ast)\leq \psi_r(x^\ast)\frac{g(x^\ast)}{\psi_r(x^\ast)} = g(x^\ast)
$$
proving that $x^\ast\in \Gamma$.
\end{proof}
Next we establish a result used to verify that a candidate strategy is optimal.
This is essentially Corollary on p. 124 in \cite{Sh}. We present the proof for readability and completeness.
\begin{lemma}\label{apu}
Let $A\subset \mathcal{I}$ be a nonempty Borel subset of $\mathcal{I}$ and $\tau_A:=\inf\{t\geq 0:X_t\in A\}$. Assume that the function
$$
\hat{V}(x):=\mathbb{E}_x\left[\textrm{e}^{-r\tau_A}g(X_{\tau_A})\right]
$$
is $r$-excessive and dominates $g$. Then, $V=\hat{V}$ and $\tau_A$ is
an  optimal stopping time. Moreover, $\tau_A$ is
finite almost surely.
\end{lemma}
\begin{proof}
Clearly, $\tau_A<\infty$ almost surely since $X$ is recurrent and $A$ is nonempty.
By the definition of $V$ it holds for all $x$
$$
V(x)=\sup_{\tau\in \mathcal{T}} \mathbb{E}_x\left[\textrm{e}^{-r\tau}g(X_\tau)\right]\geq \mathbb{E}_x\left[\mathrm{e}^{-r\tau_A}g(X_{\tau_A})\right] =\hat{V}(x).
$$
On the other hand,  $\hat{V}$ being an $r$-excessive majorant of $g$ yields
$$
V(x)=\sup_{\tau\in \mathcal{T}} \mathbb{E}_x\left[\mathrm{e}^{-r\tau}g(X_\tau)\right]\leq \sup_{\tau\in \mathcal{T}} \mathbb{E}_x\left[\mathrm{e}^{-r\tau}\hat{V}(X_\tau)\right]\leq \hat{V}(x).
$$
Consequently, $V=\hat{V}$ and $\tau_A$ is an optimal stopping time.
\end{proof}

In many optimal stopping problems the set $A$ appearing in Lemma
\ref{apu} turns out to be $\Gamma$ explaining the terminology
"stopping set" for $\Gamma$. This is also the case in our subsequent
analysis where we establish conditions under which the optimal
stopping rule equals  $\tau_{\,\Gamma}$.

\section{Main Results}

Typically optimal stopping problems of the type \eqref{osp} can be investigated quite efficiently by relying on variational inequalities and approaches utilizing the
differential operator associated with the generator of the underlying diffusion. Unfortunately, the use of those approaches for SBM is challenging due to the extra drift component involving a local time term at the skew point, see SDE \eqref{skew}. In order to circumvent this problem, we first focus on the general properties of $r$-excessive functions and characterize general conditions under which the skew point (i.e. the origin) is in the continuation region.
\begin{proposition}\label{l1}
Assume that either $0\leq g'(0-) < g'(0+)$ or $0<g'(0-) \leq g'(0+)$. Then, for SBM with $\beta>1/2$ the state $0$ is for all $r>0$ in the continuation region $C=\{x:V(x)>g(x)\}$ .
\end{proposition}
\begin{proof}
Since $\psi_r$ and $\varphi_r$ are differentiable everywhere with
respect to the scale function $S$ it follows that any $r$-excessive
function $h$ has the left and the right
%It is an elementary exercise to verify that
scale derivatives $d^-h/dS$ and $d^+h/dS$, respectively, and these satisfy for all $x$ (cf. Corollary 3.7 in \cite{salminen})
\begin{align}
\frac{d^-h}{dS}(x)\geq \frac{d^+h}{dS}(x).\label{scde}
\end{align}
Let $V$ be the value function defined in \eqref{osp} and recall that $V$ is the smallest $r$-excessive majorant of $g$. Assume now that
$0\in \Gamma$. Then $V(0)=g(0)$ and since $V(x)\geq g(x)$ for all $x\in \mathbb{R}$ we have for $\delta>0$
$$
\frac{V(0)-V(-\delta)}{S(0)-S(-\delta)}\leq \frac{g(0)-g(-\delta)}{S(0)-S(-\delta)}.
$$
Letting $\delta\downarrow 0$ yields
$$
\frac{d^-V}{dS}(0)\leq (1-\beta)g'(0-).
$$
Similarly, for $\delta>0$
$$
\frac{V(\delta)-V(0)}{S(\delta)-S(0)}\geq \frac{g(\delta)-g(0)}{S(\delta)-S(0)}
$$
leading, when letting $\delta\downarrow 0,$ to
$$
\frac{d^+V}{dS}(0)\geq \beta g'(0+).
$$
Therefore, using the assumptions on $g$,
$$
\frac{d^-V}{dS}(0)-\frac{d^+V}{dS}(0)\leq (1-\beta)g'(0-) - \beta g'(0+) \leq (1-2\beta)g'(0+)<0
$$
since $\beta>1/2$. But this contradicts \eqref{scde} and, hence, $0\not\in\Gamma$.
\end{proof}

\begin{remark}
1. In the proof of Proposition \ref{l1} we do not rely on particular properties of SBM and, therefore, the conclusions can be extended to all appropriately defined general skew diffusions.\\
2. The conclusions of Proposition \ref{l1} could alternatively be proved by investigating the behavior of the ratio
$$
u_\lambda(x):=\frac{g(x)}{\lambda\psi_r(x)+(1-\lambda)\varphi_r(x)},
$$
where $\lambda\in[0,1]$. By Theorem 2.1 in \cite{ChIr} $0\in\Gamma$ if and only if there exists a $\lambda\in[0,1]$ such that $0\in\argmax\{u_\lambda(x)\}$. Assuming that this is the case implies that $u_\lambda'(0+)\leq 0 \leq u_\lambda'(0-)$ which can be shown to coincide with the requirement $\beta g'(0+)\leq (1-\beta)g'(0-)$. Noticing that this inequality cannot be satisfied under the conditions of Proposition \ref{l1} demonstrates that $0\in C$ as claimed.
\end{remark}

Proposition \ref{l1} essentially states that if the exercise payoff is increasing in some small open neighborhood of the origin, then the skew point is always included into the continuation region. Put somewhat differently, the directional predictability of the underlying process generates incentives to wait in a neighborhood of the skew point whenever the exercise reward is locally increasing at the state where the underlying process has more tendency to move upwards instead of moving downwards. Since upward movements are in the present setting more favorable from the perspective of the decision maker, waiting becomes optimal even in cases where exercising would be optimal in the absence of skewness. This is an interesting and nontrivial property generated by the singularity of the process at the origin.

The key comparative static properties of the value and optimal exercise strategy are given in the following
\begin{proposition}\label{comparative}
The value function $V$ is non-decreasing as a function of $\beta$ and non-increasing as a function of $r$. Consequently, higher skewness (discounting) expands (contracts) or leaves unchanged the continuation region. In particular, the value function of the OSP for SBM with $\beta>1/2$ dominates the value of the corresponding OSP for standard BM $\{W_t\}_{t\geq 0}$, i.e.,
\begin{align}
V(x)\geq J(x):=\sup_{\tau\in \mathcal{T}} \mathbb{E}\left[\textrm{e}^{-r\tau}g(x+W_\tau)\right]\label{benchmark}
\end{align}
and, therefore, $\{x:J(x)>g(x)\}\subset C = \{x:V(x)>g(x)\}$.
\end{proposition}
\begin{proof}
Let $\hat{r}>r>0$ and $\tau\in \mathcal{T}$ be an arbitrary stopping time. The non-negativity of the exercise payoff $g$ then implies that
$$
\mathbb{E}_x\left[\textrm{e}^{-\hat{r}\tau}g(X_{\tau})\right] = \mathbb{E}_x\left[\textrm{e}^{-(\hat{r}-r)\tau-r\tau}g(X_{\tau})\right]\leq \mathbb{E}_x\left[\textrm{e}^{-r\tau}g(X_{\tau})\right]
$$
for all $x\in \mathbb{R}$, demonstrating that increased discounting decreases the value of the optimal policy and, consequently, does not expand the continuation region.

In order to analyze the impact of skewness on the value of the optimal timing policy, we first notice that using \eqref{skewdensity} for a  measurable function $h:\mathbb{R}\mapsto \mathbb{R}$ yields
$$
\mathbb{E}_x\left[h(X_t)\right] = \mathbb{E}\left[h(x+W_t)\right] +(2\beta-1)\int_0^\infty \frac{1}{\sqrt{2\pi t}}\textrm{e}^{-\frac{\left(|x|+y\right)^2}{2t}}(h(y)-h(-y))dy
$$
in case the expectation exist.
Consequently, for a non-decreasing $h$ it holds that
\begin{align}
\frac{\partial}{\partial \beta}\mathbb{E}_x\left[h(X_t)\right] = 2\int_0^\infty \frac{1}{\sqrt{2\pi t}}\textrm{e}^{-\frac{\left(|x|+y\right)^2}{2t}}(h(y)-h(-y))dy \geq 0.\label{partial}
\end{align}
Consider the sequence of functions $\{F_n\}_{n\geq 0}$ defined inductively (cf. \cite{Sh}, pp. 121-122) by
\begin{align*}
F_0(x) &:= g(x)\\
F_{n+1}(x) &:= \sup_{t\geq 0}\mathbb{E}_x\left[\textrm{e}^{-rt}F_n(X_t)\right].
\end{align*}
Then $F_{n+1}(x)\geq F_{n}(x)$ for all $x$ and $n$. Moreover,
$x\mapsto F_n(x)$ is non-decreasing for every $n$ since $g$ is assumed
to be non-decreasing and expectation preserves the ordering. Thus, the
increased skewness does not decrease their expected value by
\eqref{partial}. On the other hand, since $F_n$  converges pointwise
to $V$ (cf. \cite{Sh}, Lemma 5 on p. 121) we notice that the increased skewness increases or leaves unchanged $V$ and, consequently, expands the continuation region. Inequality \eqref{benchmark} follows by setting $\beta=1/2$.
\end{proof}

Proposition \ref{comparative} demonstrates that the sign of the
relationship between the increased skewness and the value of the
optimal exercise strategy is positive. This result is intuitively
clear since it essentially states that the more probable upward
excursions are, the larger is the value of waiting for more favorable
states resulting into a higher payoff. It is worth emphasizing that the
positive skewness is not needed for the positivity of the dependence
of the skewness and the value, and the conclusion is valid whenever $\beta\in [0,1]$. Proposition \ref{comparative} also shows that higher discounting accelerates rational exercise by decreasing the expected present value of future payoffs.

Before stating our main results on the single stopping boundary case,
we introduce for a differentiable function $F$
\begin{align}
\nonumber
(L_\psi
  F)(x)&:=\frac{\psi_r^2(x)}{S'(x)}\frac{d}{dx}\left[\frac{F(x)}{\psi_r(x)}\right]
\\
&=\begin{cases}\label{repinc}
\frac{1}{2}\left(\textrm{e}^{\theta x}(F'(x)-\theta F(x))+(2\beta-1)\textrm{e}^{-\theta x}(F'(x)+\theta F(x))\right), & x>0,\\
(1-\beta)\textrm{e}^{\theta x}(F'(x)-\theta F(x)), &x<0,
\end{cases}
\end{align}
and
\begin{align}
\nonumber
(L_\varphi
F)(x)&:=\frac{\varphi_r^2(x)}{S'(x)}\frac{d}{dx}\left[\frac{F(x)}{\varphi_r(x)}\right]\\
&=\begin{cases}\label{repdec}
\beta \textrm{e}^{-\theta x}(F'(x)+\theta F(x)), & x>0,\\
\frac{1}{2}\left(\textrm{e}^{-\theta x}(F'(x)+\theta F(x))-(2\beta-1)\textrm{e}^{\theta x}(F'(x)-\theta F(x))\right), &x<0.
\end{cases}
\end{align}
%\begin{align}
%(L_\psi g)(x)&:=\frac{\psi_r^2(x)}{S'(x)}\frac{d}{dx}\left[\frac{g(x)}{\psi_r(x)}\right]=\begin{cases}\label{repinc}
%\frac{1}{2}\left(\textrm{e}^{\theta x}(g'(x)-\theta g(x))+(2\beta-1)\textrm{e}^{-\theta x}(g'(x)+\theta g(x))\right), & x>0,\\
%(1-\beta)\textrm{e}^{\theta x}(g'(x)-\theta g(x)), &x<0,
%\end{cases}\\
%(L_\varphi g)(x)&:=\frac{\varphi_r^2(x)}{S'(x)}\frac{d}{dx}\left[\frac{g(x)}{\varphi_r(x)}\right]=\begin{cases}\label{repdec}
%\beta \textrm{e}^{-\theta x}(g'(x)+\theta g(x)), & x>0,\\
%\frac{1}{2}\left(\textrm{e}^{-\theta x}(g'(x)+\theta g(x))-(2\beta-1)\textrm{e}^{\theta x}(g'(x)-\theta g(x))\right), &x<0.
%\end{cases}
%\end{align}
Recall that if $F$ is  an $r$-excessive function of $X$ then $L_\psi
F$ and $L_\varphi F$ are associated with the representing measure of
$F$ (for a precise characterization and the integral representation of
excessive functions, see \cite{BS}, p. 33, \cite{salminen}  (3.3) Proposition, and
\cite{SaTa} Theorem 2.4). In the
proofs of Proposition \ref{l1b} and Proposition \ref{thm1} we use the representation theory to verify the
excessivity of the proposed value function.
\begin{proposition}\label{l1b}
(A) Let $x^\ast\in \mathcal{M}$. Then, $(-\infty,x^\ast)\setminus\mathcal{M}\subset C$.\\
\noindent (B) Assume that $\mathcal{M}=\{x^\ast\}$, where $x^\ast>0$, and that in addition to (g1) and (g2) the reward function $g$ has the following properties
\begin{itemize}
  \item[(i)] $g\in C^2([x^\ast,\infty))$ i.e. $g$ is twice continuously differentiable on $[x^\ast,\infty)$,
  \item[(ii)] $g''(x)-2rg(x)\leq 0$ for all $x\geq x^\ast$.
\end{itemize}
Then, $\tau_{x^\ast}=\inf\{t\geq 0:X_t\geq x^\ast\}$ is an optimal stopping time and the value reads as
\begin{align}
V(x) = \mathbb{E}_x\left[{\rm e}^{-r\tau_{x^\ast}}g(X_{\tau_{x^\ast}})\right] = \begin{cases}
g(x), &x\geq x^\ast,\\
\psi_r(x)\frac{g(x^\ast)}{\psi_r(x^\ast)}, &x<x^\ast.
\end{cases}\label{singval}
\end{align}
\end{proposition}
\begin{proof}
(A) Let $x\in (-\infty,x^\ast)\setminus\mathcal{M}$. It is then clear that since $x\not\in \mathcal{M}$
\begin{align}
\label{ie00}
V(x) \geq \mathbb{E}_x\left[\textrm{e}^{-r\tau_{x^\ast}}g(X_{\tau_{x^\ast}})\right] = \psi_r(x)\frac{g(x^\ast)}{\psi_r(x^\ast)}>g(x)
\end{align}
demonstrating that $x\in C$ as well.

(B) Let $\tilde{V}$ denote the proposed value function on the right hand side of \eqref{singval}. Since
\begin{align*}
V(x):=\sup_{\tau\in \mathcal{T}}\mathbb{E}_x\left[\textrm{e}^{-r\tau}g(X_\tau)\right],
\end{align*}
we find that $V\geq \tilde{V}$.

To show that $V=\tilde{V}$ we apply Lemma \ref{apu} and establish that
$\tilde{V}$ is an $r$-excessive majorant of $g$. Since $x^\ast\in
\mathcal{M}$ it is immediate that $\tilde{V}(x)\geq g(x)$ for all
$x\in \mathbb{R}$ (cf. (\ref{ie00})). To show the $r$-excessivity of $\tilde{V}$ we use the representation theory of excessive functions (cf. \cite{salminen}). Let $x_0>x^\ast$ so that $g(x_0)>0$ and define the mapping $H:\mathbb{R}\mapsto \mathbb{R}_+$
as $H(x):=\tilde{V}(x)/\tilde{V}(x_0)=\tilde{V}(x)/g(x_0)$. Moreover, let for $x\geq x_0$
\begin{align}
\label{sigma+}
\sigma_{x_0}^H((x,\infty]):=\frac{\beta \psi_r(x_0)}{\theta g(x_0)}\left(\varphi_r(x)\tilde{V}'(x)-\varphi_r'(x)\tilde{V}(x)\right)
= \frac{\psi_r(x_0)}{\theta g(x_0)}(L_\varphi g)(x)
\end{align}
and for $x\leq x_0$
\begin{align}
\label{sigma-}
\nonumber
\sigma_{x_0}^H([-\infty,x)) &:= \frac{\beta \varphi_r(x_0)}{\theta g(x_0)}\left(\psi_r'(x)\tilde{V}(x) -\psi_r(x)\tilde{V}'(x)\right) \\
&=  \begin{cases}
-\frac{\varphi_r(x_0)}{\theta g(x_0)}(L_\psi g)(x),&x\in (x^\ast,x_0],\\
0,&x\leq x^\ast.
\end{cases}
\end{align}
We now show that these definitions induce a probability measure on $[-\infty,+\infty].$ Firstly, by
the monotonicity and the non-negativity of $g$ we have that
$g'(x)+\theta g(x)\geq 0.$ Hence,  $(L_\varphi g)(x)\geq 0$ for all
$x\geq x^\ast,$ i.e., $\sigma_{x_0}^H((x,\infty])\geq 0$ for all
    $x\geq x_0.$ Moreover, from assumptions (i) and (ii)
\begin{align*}
(L_\varphi g)'(x)&=\left(g''(x)-2rg(x)\right)\varphi_r(x)\beta\leq 0
\end{align*}
for all $x\geq x^\ast$ implying that $x\mapsto
\sigma_{x_0}^H((x,\infty])$ is non-increasing.
Secondly, since $x^\ast\in \mathcal{M}$
%$x^\ast\in =\argmax\{g(x)/\psi_r(x)\}$
we have $(L_\psi g)(x^\ast)=0.$ Assumptions (i) and (ii) guarantee that
\begin{align*}
(L_\psi g)'(x)&=\left(g''(x)-2rg(x)\right)\psi_r(x)\beta\leq 0,
\end{align*}
and, therefore, $(L_\psi g)(x)\leq 0$  for all $x\geq x^\ast,$ i.e.,  $\sigma_{x_0}^H([-\infty,x))\geq 0$ for all
    $x\leq x_0,$ and  $x\mapsto
\sigma_{x_0}^H([-\infty,x))$ is non-decreasing. Thirdly, from the
  definition of the Wronskian we have that
\begin{eqnarray*}
&&\hskip -0.5cm\sigma_{x_0}^H([-\infty,x_0))+\sigma_{x_0}^H((x_0,\infty])\\
&&\hskip 2.2cm =\frac{\psi_r(x_0)}{\theta g(x_0)}\left(\frac{g'(x_0)}{S'(x_0)}\varphi(x_0)-\frac{\varphi'(x_0)}{S'(x_0)}g(x_0)\right)
-\frac{\varphi_r(x_0)}{\theta g(x_0)}\left(\frac{g'(x_0)}{S'(x_0)}\psi(x_0)-\frac{\psi'(x_0)}{S'(x_0)}g(x_0)\right) \\
&& \hskip 2.2cm =\frac{1}{\theta}\left(\frac{\psi'(x_0)}{S'(x_0)}\varphi_r(x_0)-\frac{\varphi_r'(x_0)}{S'(x_0)}\psi(x_0)\right)= 1.
\end{eqnarray*}
Combining now the three steps above and setting
$\sigma_{x_0}^H(\{x_0\})=0$ show  that $\sigma_{x_0}^H$ constitutes a
probability measure on $[-\infty, +\infty]$.   Thus, $\sigma_{x_0}^H$ induces via the Martin representation an $r$-excessive function (cf. \cite{BS}, p. 33 and \cite{salminen}) which coincides with $H$. Since $\tilde{V}(x) = \tilde{V}(x_0)H(x)$
the proposed value $\tilde{V}$ is excessive as well. Invoking Lemma \ref{apu} completes the proof.
\end{proof}

\begin{remark}
1. The conclusions of Part (B) are also valid under the weaker
assumptions:
\begin{itemize}
  \item[(i)] $g\in C^1([x^\ast,\infty)),$
  \item[(ii)]  $L_\varphi g$ and  $L_\psi g$ are non-increasing on $[x^\ast,\infty)$.
\end{itemize}

\noindent
2. In the proof of Proposition \ref{l1b} it is seen that
$\sigma^H_{x_0}$ induces a probability measure on $[-\infty,+\infty].$
In fact, $\sigma_{x_0}^H(\{-\infty\})=0$ and $\sigma_{x_0}^H(\{+\infty\})=0.$
Indeed, the first statement is immediate from (\ref{sigma-}). The
second one follows if $\lim_{x\to +\infty}(L_\varphi g)(x)=0$
(cf. (\ref{sigma+})). To verify this, recall from the proof of
Proposition \ref{l1b} that $(L_\psi g)(x)\leq 0$ for $x>x^*,$ and,
hence,
$$
g'(x) \leq \frac{\psi_r'(x)}{\psi_r(x)}g(x)= \frac{\textrm{e}^{\theta
    x}-(2\beta-1)\textrm{e}^{-\theta x}}{\textrm{e}^{\theta
    x}+(2\beta-1)\textrm{e}^{-\theta x}}\theta g(x)\leq \theta g(x).
$$
Consequently, for  $x\geq x_0$
\begin{align}\label{useful}
(L_\varphi g)(x)=\frac{\beta(g'(x) +\theta g(x))}{\mathrm{e}^{\theta x}}\leq 2\beta\theta  \mathrm{e}^{-\theta x}g(x).
\end{align}
Because $\lim_{x\rightarrow\infty}\mathrm{e}^{-\theta x}\psi_r(x)=1$
and, by assumption, $\lim_{x\rightarrow\infty} g(x)/\psi_r(x) = 0$  we
have $\sigma_{x_0}^H(\{\infty\})=\lim_{x\uparrow
  \infty}\sigma_{x_0}^H((x,\infty]) =0$, as claimed.

%here
\end{remark}

Part (A) of Proposition \ref{l1b} shows how the ratio $g/\psi_r$ can be utilized in the characterization of subsets of the continuation region. An interesting implication of these findings is that $(-\infty,\inf{\mathcal{M}})\subset C$. Hence, if the maximizing threshold $x^\ast$ of the ratio $g/\psi_r$ is negative, unique, an the exercise payoff is increasing and either differentiable or locally convex at the origin, then
the continuation region must necessarily contain both the set $(-\infty,x^\ast)$ as well as an open neighborhood of the origin. This result illustrates nicely the intricacies associated with the singularity of the underlying diffusion at the skew point. As we will later observe, this phenomenon arises even for piecewise linear reward functions.

Part (B) of Proposition \ref{l1b} in turn states a set of conditions under which the general optimal timing problem constitutes a standard single exercise boundary problem where the underlying process is stopped as soon as it hits the critical threshold $x^\ast>0$ at which the ratio $g/\psi_r$ is maximized.  The results of part (B) can naturally be extended to the case where the maximizing threshold is negative, i.e., to the case where $x^\ast < 0$. However, it is clear from Proposition \ref{l1} that in that case the exercise payoff has to be constant in a neighborhood of the skew point since otherwise the origin could not belong to the stopping set.

Our main results on the case where $x^\ast<0$ are now summarized in the following proposition.
\begin{proposition}\label{thm1}
Assume that $\mathcal{M}=\{x^\ast\}$, where $x^\ast<0$, and that in addition to conditions (g1) and (g2) the exercise payoff $g$ satisfies the conditions
\begin{itemize}
  \item[(i)] $g\in C^2([x^\ast,\infty))$,
  \item[(ii)] $(1-\beta)\theta \mathrm{e}^{-\theta x^\ast}g(x^\ast)>\beta g'(0) > 0,$
  \item[(iii)] $g''(x)-2rg(x)< -\varepsilon$ for all $x\geq x^\ast$ and some $\varepsilon>0$.
\end{itemize}
Then, the equation system
\begin{align}
\label{repre}
 \begin{cases}(L_\psi g)(x)=(L_\psi g)(y)&\\
(L_\varphi g)(x)=(L_\varphi g)(y)&\\
 \end{cases}
\end{align}
%\begin{align}
%(L_\psi g)(x)&=(L_\psi g)(y)\label{repre1}\\
%(L_\varphi g)(x)&=(L_\varphi g)(y)\label{repre2}
%\end{align}
has a unique solution $\mathbf{y}^\ast =(y_1^\ast,y_2^\ast)$ such that $\mathbf{y}^\ast\in (x^\ast,0)\times(0,\infty)$. Moreover,
$\tau^\ast=\inf\{t\geq 0:X_t\in A\}$ with $A = [x^\ast,y_1^\ast]\cup[y_2^\ast,\infty)$ is the optimal stopping time, and the value reads as
\begin{align}\label{value}
V(x) = \begin{cases}
g(x), &x\in [x^\ast,y_1^\ast]\cup[y_2^\ast,\infty),\\
g(x^\ast)\frac{\psi_r(x)}{\psi_r(x^\ast)}, &x\in (-\infty,x^\ast),\\
g(y_1^\ast)\;\mathbb{E}_x\left[\mathrm{e}^{-r\hat{\tau}_{y_1^\ast}};\hat{\tau}_{y_1^\ast}<\hat{\tau}_{y_2^\ast}\right]&\\
\hskip1cm+
g(y_2^\ast)\;\mathbb{E}_x\left[\mathrm{e}^{-r\hat{\tau}_{y_2^\ast}};\hat{\tau}_{y_2^\ast}<\hat{\tau}_{y_1^\ast}\right],&x\in(y_1^\ast,y_2^\ast),
\end{cases}
\end{align}
where
\begin{align*}
\mathbb{E}_x\left[\mathrm{e}^{-r\hat{\tau}_{y_1^\ast}};\hat{\tau}_{y_1^\ast}<\hat{\tau}_{y_2^\ast}\right]&=\frac{\varphi_r(x)\psi_r(y_2^\ast)-\psi_r(x)\varphi_r(y_2^\ast)}
{\psi_r(y_2^\ast)\varphi_r(y_1^\ast)-\varphi_r(y_2^\ast)\psi_r(y_1^\ast)}\\
\mathbb{E}_x\left[\mathrm{e}^{-r\hat{\tau}_{y_2^\ast}};\hat{\tau}_{y_2^\ast}<\hat{\tau}_{y_1^\ast}\right]&=\frac{\psi_r(x)\varphi_r(y_1^\ast)-\varphi_r(x)\psi_r(y_1^\ast)}
{\psi_r(y_2^\ast)\varphi_r(y_1^\ast)-\varphi_r(y_2^\ast)\psi_r(y_1^\ast)}.
\end{align*}
\end{proposition}
\begin{proof}
We first establish that equation system \eqref{repre}  has a unique solution $\mathbf{y}^\ast\in (x^\ast,0)\times(0,\infty)$.
In order to accomplish this task, we first observe that \eqref{repre} can be re-expressed by using \eqref{repinc} and \eqref{repdec} as
\begin{align}
\label{repre2}
 \begin{cases}(1-\beta)(q_1(x)+q_2(x))= \beta (q_1(y)+q_2(y))&\\
q_1(x)-q_2(x)= q_1(y)-q_2(y),&\\
 \end{cases}
\end{align}
%\begin{align}
%(1-\beta)(q_1(x)+q_2(x))&= \beta (q_1(y)+q_2(y))\label{eq1gen}\\
%q_1(x)-q_2(x)&= q_1(y)-q_2(y),\label{eq2gen}
%\end{align}
where  $q_1(x):=\textrm{e}^{\theta x}(g'(x)-\theta g(x))$ and
$q_2(x):=\textrm{e}^{-\theta x}(g'(x)+\theta g(x))$. Consider now the
behavior of the functions $h_1:=q_1+q_2$ and  $h_2:=q_1-q_2.$
%\begin{align}
%h_1(x) &:=q_1(x)+q_2(x)\label{f1}\\
%h_2(x) &:=q_1(x)-q_2(x).\label{f2}
%\end{align}
Since  $x^\ast<0$ and $(L_\psi g)(x^*)=0$  it follows from
(\ref{repinc})
%taa
that $q_1(x^\ast)=0$ and, hence, $h_1(x^\ast)=-h_2(x^\ast)=\mathrm{e}^{-\theta x^\ast}2\theta g(x^\ast) >0$. Moreover, $h_1(0)=2g'(0)>0$, $h_2(0)=-2\theta g(0)<0$, and
\begin{align}
h_1'(x) &=(\mathrm{e}^{\theta x}+\mathrm{e}^{-\theta x})(g''(x)-2rg(x))\label{df1}\\
h_2'(x) &=(\mathrm{e}^{\theta x}-\mathrm{e}^{-\theta x})(g''(x)-2rg(x)).\label{df2}
\end{align}
Our assumption (iii) guarantees that $h_1'(x)<0$ for all $x>x^\ast$. In a completely analogous fashion we find that
$
h_2'(x)<0
$ for $x>0$ and $h_2'(x)>0$ for $x\in (x^\ast,0)$. Moreover, if $x>z>0$ then applying the standard mean value theorem yields
\begin{align*}
h_1(x)-h_1(z) &= \int_z^x(\mathrm{e}^{\theta t}+\mathrm{e}^{-\theta t})(g''(t)-2rg(t))dt\\ &=\frac{(g''(\xi)-2rg(\xi))}{\theta}\left[(\mathrm{e}^{\theta x}-\mathrm{e}^{-\theta x})-(\mathrm{e}^{\theta z}-\mathrm{e}^{-\theta z})\right]
\end{align*}
demonstrating that $\lim_{x\rightarrow\infty}h_1(x)=-\infty$. In an analogous way we find that $\lim_{x\rightarrow\infty}h_2(x)=-\infty$ as well.
Consider now for a given $x\in[x^\ast,0]$ equation $h_2(\tilde{y}_x)=h_2(x)$ where $\tilde{y}_x\in[0,\infty)$. The continuity of $h_2(x)$ at the origin implies that for $x=0$ we have $\tilde{y}_0=0$. Utilizing \eqref{df2}, in turn, implies that for all $x\in (x^\ast,0)$ there is a unique $\tilde{y}_x\in (0, \tilde{y}_{x^\ast})$ satisfying $h_2(\tilde{y}_x)=h_2(x)$ (since $h_2(x)\downarrow -\infty$ as $x\uparrow\infty$). Implicit differentiation yields
$$
\tilde{y}_x' = \frac{h_2'(x)}{h_2'(\tilde{y}_x)} < 0.
$$
Consider next  for a given $x\in[x^\ast,0]$ equation $l(x)=l(\hat{y}_x)$, where
$$
l(x)=\begin{cases}
\beta h_1(x),&x>0,\\
(1-\beta) h_1(x),&x<0,
\end{cases}
$$
and $\hat{y}_x\in[0,\infty)$.
The monotonicity of $h_1(x)$ implies that $l(x)$ is monotonically decreasing on $(x^\ast,0)\cup(0,\infty)$. Moreover, since
$$
l(0+)-l(0-) =\beta h_1(0)-(1-\beta) h_1(0) = (2\beta-1)2g'(0)>0,
$$
$$
l(x^\ast)-l(0+) = 2((1-\beta)\theta \mathrm{e}^{-\theta x^\ast}g(x^\ast)-\beta g'(0)) > 0,
$$
and $l(x)\downarrow -\infty$ as $x\uparrow\infty$
we notice that there exists necessarily a unique $\hat{x}\in (x^\ast,0)$ such that $l(\hat{x})=l(0+)$ and, consequently, such that $\hat{y}_{\hat{x}}=0$. On the other hand, since $l(x)<0$ for $x>l^{-1}(0)$ we notice that there is a unique $\hat{y}_0\in (0,l^{-1}(0))$ such that $l(\hat{y}_0)=l(0-)$. Moreover, implicit differentiation yields
$$
\hat{y}_x'=\frac{(1-\beta)h_1'(x)}{\beta h_1'(\hat{y}_x)} > 0.
$$
Combining these findings show that $\tilde{y}_0=0<\hat{y}_0$ and
$\tilde{y}_{x^\ast}>\tilde{y}_{\hat{x}}>0=\hat{y}_{\hat{x}}$. The
continuity and the monotonicity of the solution curves $x\mapsto \tilde{y}_x$ and
$x\mapsto \hat{y}_x,\  x\in(x^*,0)$ then proves that they have a unique interception point
$x^{**}\in(\hat{x},0)$ such that $\tilde{y}_{x^{**}}=\hat{y}_{x^{**}}$
and, consequently, such that \eqref{repre2} holds.

We now prove that \eqref{value} constitutes the value and $\tau^\ast$
the optimal stopping strategy of \eqref{osp}. To this end, let
$\tilde{V}$ denote the proposed value function on the right hand side
of \eqref{value} with $y^*_1:=x^{**}$ and $y^*_2:=\tilde{y}_{x^{**}}=\hat{y}_{x^{**}}.$ It  is again clear that $V\geq \tilde{V}$. In order to prove the opposite inequality, we first notice that $\tilde{V}$ is continuous and non-negative. To demonstrate that $\tilde{V}$ is $r$-excessive, we let $x_0>y_2^\ast$ and define the mapping $\hat{H}:\mathbb{R}\mapsto \mathbb{R}_+$
as $\hat{H}(x):=\tilde{V}(x)/\tilde{V}(x_0)=\tilde{V}(x)/g(x_0)$. As in the proof of Proposition \ref{l1b}, define for $x\geq x_0$
\begin{align*}
\sigma_{x_0}^{\hat{H}}((x,\infty]):= \frac{\psi_r(x_0)}{\theta g(x_0)}(L_\varphi g)(x)
\end{align*}
and for $x\leq x_0$
\begin{align*}
\sigma_{x_0}^{\hat{H}}([-\infty,x)) &:= \frac{\varphi_r(x_0)}{\theta
    g(x_0)}\left(\tilde V(x)\frac{d^-\psi_r}{dS}(x) -\psi_r(x)\frac{d^-\tilde
    V}{dS}(x)\right) \\
&=  \begin{cases}
-\frac{\varphi_r(x_0)}{\theta g(x_0)}(L_\psi g)(x),&x\in (y_2^\ast,x_0],\\
-\frac{\varphi_r(x_0)}{\theta g(x_0)}(L_\psi g)(y_1^\ast), &x\in(y_1^\ast,y_2^\ast],\\
-\frac{\varphi_r(x_0)}{\theta g(x_0)}(L_\psi g)(x), &x\in(x^\ast,y_1^\ast],\\
0,&x\in(-\infty, x^\ast],
\end{cases}
\end{align*}
where the identity $(L_\psi g)(y_1^\ast)=(L_\psi g)(y_2^\ast)$ is used.
We now show that these definitions induce a probability measure
%$\sigma_{x_0}^{\hat{H}}$
on $[-\infty,+\infty]$. Firstly,
the monotonicity and the non-negativity of the exercise payoff $g$
imply that $g'(x)+\theta g(x)>0$ and, therefore,  from \eqref{repdec}
$(L_\varphi g)(x)\geq 0$ for all $x\geq y_2^\ast,$ i.e.,
$\sigma_{x_0}^{\hat{H}}((x,\infty])\geq 0$ for $x\geq x_0.$ Moreover,
    $$(L_\varphi g)'(x) = \beta (g''(x)-2rg(x))\varphi_r(x)<0$$ for
    all $x\in[x_0,\infty)$ implying that $x\mapsto
      \sigma_{x_0}^{\hat{H}}((x,\infty])$ is non-increasing.
Secondly, since  $(L_\psi g)(x^\ast)=0$ and
\begin{align}
(L_\psi g)'(x) = \begin{cases}
\beta (g''(x)-2rg(x))\psi_r(x),&x\in(0,\infty),\\
(1-\beta) (g''(x)-2rg(x))\psi_r(x),&x\in(x^\ast,0),
\end{cases}\label{mon}
\end{align}
it is seen by applying assumption (iii)  that  $L_\psi g$ is decreasing and negative on
$(x^\ast,\infty).$ Consequently,  $x\mapsto
      \sigma_{x_0}^{\hat{H}}([-\infty,x))$ is non-negative and
        non-decreasing for $x\leq x_0.$ Thirdly, we should check that
$$
      \sigma_{x_0}^{\hat{H}}([-\infty,x_0))+ \sigma_{x_0}^{\hat{H}}([x_0,+\infty])=1,
$$
but this follows similarly as in the proof of Proposition \ref{l1b}
exploiting the Wronskian relationship. This concludes the proof that
$\sigma_{x_0}^{\hat{H}}$ constitutes a probability measure on
$[-\infty,+\infty].$  The probability measure  $\sigma_{x_0}^{\hat{H}}$ induces via the Martin representation an $r$-excessive function (cf. \cite{BS}, p. 33 and \cite{salminen}) which coincides with $\hat{H}$. Since $\tilde{V}(x) = \tilde{V}(x_0)\hat{H}(x)$ we find that
the proposed value $\tilde{V}(x)$ is $r$-excessive as well.

It remains to prove that $\tilde{V}$ dominates the exercise payoff $g$. It is clear that $\tilde{V}\geq g$ for all $x\in(-\infty, y_1^\ast]\cup[y_2^\ast,\infty)$.
It is, thus, sufficient to analyze the difference
$\Delta(x):=\tilde{V}(x)-g(x)$ on $(y_1^\ast, y_2^\ast)$. Notice that
$\Delta(y_1^\ast)=\Delta(y_2^\ast)=0$. Applying formula (3.4) in \cite{salminen} where we choose $x_0=y_2^\ast$ results in
\begin{align*}
\frac{\tilde{V}(x)}{\tilde{V}(x_0)}=\frac{\sigma^{\hat{H}}_{x_0}([-\infty,x))}{\varphi_r(y_2^\ast)}\varphi_r(x)+
\frac{\sigma^{\hat{H}}_{x_0}((x,\infty])}{\psi_r(y_2^\ast)}\psi_r(x),\quad x\in(y_1^\ast,y_2^\ast).
\end{align*}
Since $\sigma^{\hat{H}}_{x_0}([y^*_1,y^*_2])=0$ this expression
simplifies and yields
\begin{align*}
\tilde{V}(x)&=g(y_2^\ast)\left(-\frac{\varphi_r(y_2^\ast)}{\theta g(y_2^\ast)}(L_\psi g)(y_2^\ast)\right)\frac{\varphi_r(x)}{\varphi_r(y_2^\ast)}+
g(y_2^\ast)\left(\frac{\psi_r(y^\ast)}{\theta g(y_2^\ast)}(L_\varphi g)(y_2^\ast)\right)\frac{\psi_r(x)}{\psi_r(y_2^\ast)}\\
&= -\frac{(L_\psi g)(y_2^\ast)}{\theta }\varphi_r(x)+
\frac{(L_\varphi g)(y_2^\ast)}{\theta}\psi_r(x).
\end{align*}
Moreover, utilizing \eqref{mon}, assumption (iii),  and noticing that
\begin{align*}
\frac{d}{dx}\left[\frac{\Delta(x)}{\psi_r(x)}\right] = \frac{S'(x)}{\psi_r^2(x)}\left((L_\psi g)(y_i^\ast)-(L_\psi g)(x)\right)
\end{align*}
for $x\in(y_1^\ast,0)\cup(0,y_2^\ast)$ show that $\Delta/\psi_r$ is increasing on $(y_1^\ast,0)$ and, consequently, that $\Delta(x)>0$ for all $x\in(y_1^\ast,0)$. In an completely analogous fashion, we find that $\Delta(x)/\psi_r(x)$ is decreasing for $x\in(0,y_2^\ast)$ and, therefore, that  $\Delta(x)>0$ for all $x\in(0, y_2^\ast)$ as well. The continuity of $\Delta$ then proves that $\Delta(x)>0$ for all $x\in(y_1^\ast,y_2^\ast)$ and, consequently, that the proposed value $\tilde{V}$ dominates the exercise payoff $g$.

We may now evoke Lemma \ref{apu} to complete the proof of the proposition,

\end{proof}

\begin{remark}
The conclusions of Proposition \ref{thm1} are derived from the general properties of excessive mappings and their representing measures and as such do not require detailed process specific information besides the singularity at the skew point and the generator of the driving process. In that respect, the developed proof applies even under more general circumstances than in the SBM setting.
\end{remark}

As the proof of Proposition \ref{thm1} indicates, there are circumstances under which the problem can be reduced into a two boundary problem where the lower boundary $x^\ast=y_1^\ast$ constitutes a tangency point of the value. A set of conditions under which this observation is true are stated in the following corollary.
\begin{corollary}\label{tangeeraus}
Assume that $\mathcal{M}=\{x^\ast,y_2^\ast\}$, where $x^\ast<0<y_2^\ast$. Assume also that conditions (i) - (iii) of Proposition \ref{thm1} are satisfied. Then
$\Gamma = \{x^\ast\}\cup[y_2^\ast,\infty)$, $C=(-\infty,x^\ast)\cup(x^\ast,y_2^\ast)$, and the value is
\begin{align}
\label{tan11}
V(x) = \begin{cases}
g(x), &x\in \{x^\ast\}\cup[y_2^\ast,\infty),\\
\psi_r(x)\frac{g(y_2^\ast)}{\psi_r(y_2^\ast)}, &x\in (-\infty,x^\ast)\cup(x^\ast,y_2^\ast).
\end{cases}
\end{align}
\end{corollary}
\begin{proof}
The statement is a direct implication of part (A) of Proposition \ref{l1b} and Proposition \ref{thm1}.
\end{proof}

\section{Explicit Illustration}

Our objective is now to illustrate the main results in Section 4 explicitly by assuming that the exercise reward reads as $g(x):=(x+K)^+$ with $K> 0$. Recall that $\mathcal{M}$ denotes the set of maximum points of the ratio $g/\psi_r$, cf. \eqref{M}.
Our main result on the value and the optimal stopping strategy are presented in the following:
\begin{proposition}
For all $\beta\in(1/2,1)$ and $K>0$ there is a unique critical discount rate $\hat{r}=\hat{r}(\beta,K)$ satisfying the identity
\begin{align}
\beta+ \beta\ln\left(\beta + \sqrt{\beta^2 +(2\beta-1){\rm{e}}^{2(\sqrt{2\hat{r}} K-1)}}\right) = \sqrt{\beta^2+(2\beta-1){\rm{e}}^{2(\sqrt{2\hat{r}} K-1)}}.\label{critical}
\end{align}
Moreover, $\hat{r}$ is increasing as a function of $\beta$.\\
(A) Assume that $r<\hat{r}$. Then, $\mathcal{M}=\{x^\ast\}$ with $x^\ast>0$.
The optimal stopping strategy is $\tau^\ast=\inf\{t\geq 0: X_t\geq x^\ast\}$ and the value
is as in \eqref{singval}.\\
(B) Assume that $r=\hat{r}$. Then $\mathcal{M}=\{x_1^\ast,x^\ast\},$
 where $x^\ast>0$ and
\begin{align}
x_1^\ast = \frac{1}{\theta}-K<0.\label{lower}
\end{align}
The optimal stopping strategy is
$\tau^\ast=\inf\{t\geq 0: X_t\in\{x_1^\ast\}\cup [x^\ast,\infty)\}$ and the value
is  as in \eqref{tan11}.\\
(C) Assume that $r>\hat{r}$. Then,  $\mathcal{M}=\{x_1^\ast\}$ where $x_1^\ast$ is as given in \eqref{lower}.
The optimal stopping strategy is
$\tau^\ast=\inf\{t\geq 0: X_t\in[x_1^\ast,y_1^\ast]\cup
[y_2^\ast,\infty)\}$, where $(y_1^\ast,y_2^\ast)\in
  (x_1^\ast,0)\times(0,\infty)$ constitute the unique solution of the equation system \eqref{repre}, and the value
is as in \eqref{value}.
\end{proposition}
\begin{proof}
In what follows we will show that the three different cases (A)-(C) appearing above and corresponding to the cases characterized in Proposition \ref{l1b}, Corollary \ref{tangeeraus}, and Proposition \ref{thm1}
arise depending on the precise magnitude of the key parameters $\beta, r$ and $K$. We start by proving that for any $\beta\in(1/2,1)$ and $K>0$ equation \eqref{critical} has a unique solution $\hat{r}$.
To this end, fix $K>0$ and consider for $\theta>0$ and $\beta\in[1/2,1]$ the function
$$
C(\theta,\beta):=\beta+ \beta\ln\left(\beta + \sqrt{\beta^2 +(2\beta-1)\textrm{e}^{2(\theta K-1)}}\right) - \sqrt{\beta^2+(2\beta-1)\textrm{e}^{2(\theta K-1)}}.
$$
Standard differentiation yields
\begin{align}
C_\theta(\theta,\beta)&=-\frac{(2\beta-1)\textrm{e}^{2(\theta K-1)}K}{\beta +\sqrt{\beta^2 +(2\beta-1)\textrm{e}^{2(\theta K-1)}}}<0\label{herkkyys1}\\
% C_{\theta\theta}(\theta,\beta)&=-\frac{(2\beta-1)\textrm{e}^{2(\theta K-1)}K^2}{\sqrt{\beta^2 +(2\beta-1)\textrm{e}^{2(\theta K-1)}}}<0\\
C_\beta(\theta,\beta)&= 1+\ln\left(\beta + \sqrt{\beta^2 +(2\beta-1)\textrm{e}^{2(\theta K-1)}}\right)+\frac{\beta - \sqrt{\beta^2 +(2\beta-1)\textrm{e}^{2(\theta K-1)}}}{2\beta-1}\label{herkkyys2}
\end{align}
Consequently, from \eqref{herkkyys1}, $C$ is monotonically decreasing as a function of $\theta$. In particular, for all $\beta\in(1/2,1)$ we have
\begin{align*}
C(1/K,\beta)&=\beta+ \beta\ln\left(\beta + \sqrt{\beta^2 +2\beta-1}\right) - \sqrt{\beta^2+2\beta-1}>0,\\
C(\theta^\ast,\beta)&=\beta\left(\ln\left(\frac{\beta}{1-\beta}\right) -\frac{2\beta-1}{1-\beta}\right)<0,
\end{align*}
where
\begin{align}
\theta^\ast=\left(1+\ln\left(\frac{\beta}{1-\beta}\right)\right)\frac{1}{K} > \frac{1}{K}.\label{crittheta}
\end{align}
Invoking the monotonicity and the continuity of $C$ as a function of $\theta$ shows that equation  \eqref{critical} has a unique solution, as claimed.

Next we show that $\beta\mapsto\hat{r}(\beta)$ is increasing. To see that this is indeed the case, consider the function $\hat{\theta} :=\sqrt{2\hat{r}}$ and observe that implicit differentiation of equation
$C(\hat{\theta},\beta)=0$ yields
\begin{align}
\hat{\theta}'=-\frac{C_\beta(\hat{\theta},\beta)}{C_\theta(\hat{\theta},\beta)}.\label{suunta}
\end{align}
Since  $C_{\theta}< 0$ by \eqref{herkkyys1}, it is sufficient to study the sign of $C_\beta$ along the solution curve $\beta\mapsto\hat{\theta}(\beta)$.
Since $\hat{\theta} < \theta^\ast$ and $\hat{\theta} K-1<\ln\left(\beta/(1-\beta)\right)$, we have from \eqref{herkkyys2} using the identity $C(\hat{\theta},\beta)=0$ that
\begin{align}
C_\beta(\hat{\theta},\beta)= \frac{1-\beta}{\beta(2\beta-1)}\left(\frac{\beta^2}{1-\beta} - \sqrt{\beta^2 +(2\beta-1)\textrm{e}^{2(\hat{\theta} K-1)}}\right)>0.\label{sensitivity}
\end{align}
Therefore, from \eqref{suunta} and \eqref{herkkyys1} it follows that $\hat{\theta}'>0$ and, hence, $\hat{r}$ is increasing.

We now proceed to proving (A)-(C). From our general analysis we known that we should consider the maximum points of the function
$$
u_r(x) := \frac{(x+K)^+}{\psi_r(x)} = \begin{cases}
 \displaystyle{\frac{2\beta(x+K)}{\textrm{e}^{\theta x}+\left(2\beta-1\right)\textrm{e}^{-\theta x}}},&x>0,\\
\textrm{e}^{-\theta x}(x+K)^+,&x\leq 0.
\end{cases}
$$
Standard differentiation yields
\begin{align*}
l(x)&:=\psi_r^2(x)u_r'(x)\\ &= \begin{cases}
 \frac{1}{2\beta}\textrm{e}^{\theta x}(1-\theta (x+K))+\left(1-\frac{1}{2\beta}\right)\textrm{e}^{-\theta x}(1+\theta(x+K)),&x>0,\\
\textrm{e}^{\theta x}(1-\theta(x+K)),&-K < x < 0,\\
0,&x<-K.
\end{cases}
\end{align*}
We immediately notice the following $$l(0-)=1-\theta K,\quad l(0+)=1-\left(\frac{1}{\beta}-1\right)\theta K, \quad l(0+)-l(0-)=\frac{2\beta-1}{\beta}\theta K>0,$$ and
$\lim_{x\rightarrow\infty}l(x)=-\infty.$ Moreover, since for $x>-K$
$$
l'(x)= \begin{cases}
 -\theta^2(x+K)\left(\frac{1}{2\beta}\textrm{e}^{\theta x}+\left(1-\frac{1}{2\beta}\right)\textrm{e}^{-\theta x}\right),&x>0,\\
-\textrm{e}^{\theta x}\theta^2(x+K),&-K<x<0,
\end{cases}
$$
two different configurations may arise depending on the precise values of $\theta, \beta,$ and $K$. First, if $\theta K\leq 1$, then $l(0-)\geq0$ and the monotonicity
of $l$ guarantees that $u_r$ attains a unique global maximum
at  $x^\ast>0$ satisfying the ordinary first order condition $u_r'(x^\ast)=0$ which is equivalent with
\begin{align}\label{optc1}
\textrm{e}^{\theta x^\ast}(1-\theta (x^\ast+K))+\left(2\beta-1\right)\textrm{e}^{-\theta x^\ast}(1+\theta(x^\ast+K))=0.
\end{align}
This case corresponds to the one characterized in part (B) of Proposition \ref{l1b} and, hence, proves claim (A) when  $\theta K\leq 1$.

Second, if $\theta K>1$ then $l(-K)=\textrm{e}^{-\theta K}>0$ and the monotonicity of $l$ on $(-K,0)$ guarantees that $u_r$ attains
a local maximum at the point
$$
x_1^\ast = \frac{1}{\theta}-K<0.
$$
If
$l(0+)=1-\left(\frac{1}{\beta}-1\right)\theta K>0$, then $u_r$ attains a local maximum at the threshold  $x^\ast>0$ satisfying \eqref{optc1} as well. However, if $l(0+)\leq 0$, then the monotonicity of $l$ implies that $x_1^\ast$ constitutes a global maximum point of $u_r$ and $\mathcal{M}=\{x_1^\ast\}$.
Hence, in the case where $l(0+)>0$ the set $\mathcal{M}$ has at most two points.
In order to determine the parameter values for which $\mathcal{M}=\{x_1^\ast,x^\ast\}$ we consider the equation
\begin{align}
u_r(x^\ast)-u_r(x_1^\ast) =0.\label{equil}
\end{align}
Since $u_r'(x^\ast)=u_r'(x_1^\ast)=0$ it holds that
$u_r(x^\ast)=1/\psi_r'(x^\ast)$ and $u_r(x_1^\ast)=1/\psi_r'(x_1^\ast)$. Hence, \eqref{equil} is equivalent with
\begin{align}
 \frac{1}{\psi_r'(x^\ast)}-\frac{1}{\psi_r'(x_1^\ast)}=\frac{2\beta}{\theta(\textrm{e}^{\theta x^\ast}-(2\beta-1)\textrm{e}^{-\theta x^\ast})}-\frac{1}{\theta}\textrm{e}^{\theta K-1} =0.\label{equilb}
\end{align}
Consequently, $\mathcal{M}=\{x_1^\ast, x^\ast\}$ with $x^\ast>0$ as in \eqref{optc1} if and only if $x^\ast$ satisfies also \eqref{equilb}, which is  equivalent with
\begin{align}
\textrm{e}^{2\theta{x^{\ast}}}-2\beta \textrm{e}^{1-\theta K}\textrm{e}^{\theta {x^{\ast}}}-(2\beta-1)=0\label{polynomial}
\end{align}
implying that
\begin{align}
{x^{\ast}}= \frac{1}{\theta}\ln\left(\beta \textrm{e}^{1-\theta K} +
\sqrt{\beta^2 \textrm{e}^{2(1-\theta K)}+(2\beta-1)}\right).
\label{1vaihe}
\end{align}
Substituting the expression for $2\beta-1$ obtained from \eqref{polynomial} into \eqref{optc1} yields
\begin{align}
{\rm{ e}}^{\theta x^\ast}=\beta {\rm e}^{1-\theta K} (1+\theta(x^\ast+K)).\label{sieva}
\end{align}
By applying \eqref{1vaihe} in \eqref{sieva} we conclude that $\mathcal{M}=\{x_1^\ast,x^\ast\}$ if and only if $\beta\in[1/2,1]$ and $\theta > 0$ are such that
$C(\beta,\theta)=0$, as claimed. This proves case (B), and also (A) and (C) since the value is a non-increasing function of $r$.
\end{proof}
\begin{remark}
For $\beta=1$ equation (\ref{critical}) with $\hat\theta:=\sqrt{2\hat
  r}$ reads as
$$
1+ \ln\left(1 + \sqrt{1 +{\rm{e}}^{2(\hat{\theta} K-1)}}\right) = \sqrt{1+{\rm{e}}^{2(\hat{\theta} K-1)}},
$$
and the unique solution is given by $\hat{\theta}K\approx
1.64132.$ Notice that $\beta\mapsto\theta(\beta)$ being increasing the
limit of $\theta(\beta)$ as $\beta\downarrow 1/2$ exists. As
$\beta\downarrow 1/2$ then necessarily  $x^*$ in (\ref{1vaihe}) tends
to 0. Therefore, $\lim_{\beta\downarrow 1/2}\theta(\beta)=1/K.$
Consequently, the critical parameter boundary
$\beta\mapsto\theta(\beta)$ is an increasing function connecting the
extremal points $(1/2,1/K)$ and $(1,1.64132/K)$. This is illustrated
in Figure \ref{sallittu} when $K=1.$
\end{remark}

\begin{figure}[!ht]
\begin{center}
\includegraphics[width=0.6\textwidth]{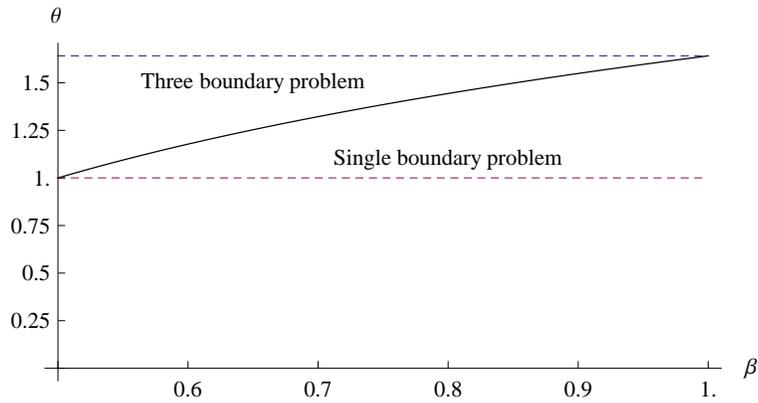}
\caption{\small Critical Boundary; with $K=1$}\label{sallittu}
\end{center}
\end{figure}

The optimal boundaries associated with the optimal exercise strategies are illustrated as functions of the skewness parameter $\beta$ in Figure \ref{2reunaongelma} under the assumptions that
$K=1$ and $r=0.95$. As is clear from the figure, the considered stopping problem constitutes a three-boundary problem as long as the skewness parameter $\beta$ remains below the critical level $\beta^\ast$ which under our parameter assumptions is $\beta^\ast\approx 0.7445$. As soon as skewness exceeds this critical level, the problem becomes a single boundary problem, where the decision maker waits until the underlying hits the upper threshold maximizing the ratio $(x+K)^+/\psi_r(x)$. The reason for this observation is clear: for sufficiently low values of $\beta$ the attainable intertemporal gains accrued by waiting and postponing the timing decision further into the future exceed the return accrued by exercising immediately in a neighborhood of the origin. As the skewness parameter increases, more and more of the excursions are expected to end to the positive side, thus increasing the incentives to wait for higher payoffs.

\begin{figure}[!ht]
\begin{center}
\includegraphics[width=0.6\textwidth]{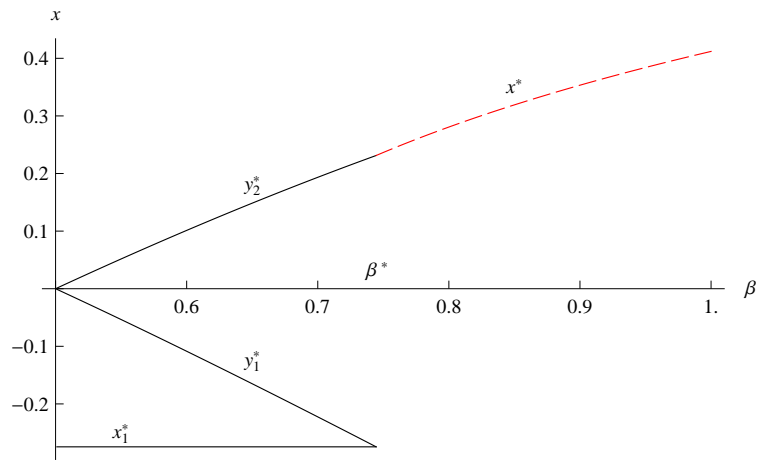}
\caption{\small Optimal Stopping Boundaries; with $K=1$ and $r=0.95$.}\label{2reunaongelma}
\end{center}
\end{figure}

The optimal boundaries associated with the optimal exercise strategies are, in turn, illustrated as functions of the parameter $\theta$ in Figure \ref{2reunaongelmab} under the assumptions that
$K=1$ and $\beta=0.55$. In contrast with the effect of the skewness parameter $\beta$, higher discounting accelerates optimal timing and, thus, decreases the incentives to wait. Accordingly, we now notice from Figure \ref{2reunaongelmab} that the considered problem constitutes a single boundary problem only as long as the discount rate is lower than the critical level $\hat{r}\approx 0.5983$. Above this critical level waiting for for future potentially higher payoffs is no longer optimal at all states and the optimal exercise strategy becomes a three-boundary stopping rule.

\begin{figure}[!ht]
\begin{center}
\includegraphics[width=0.6\textwidth]{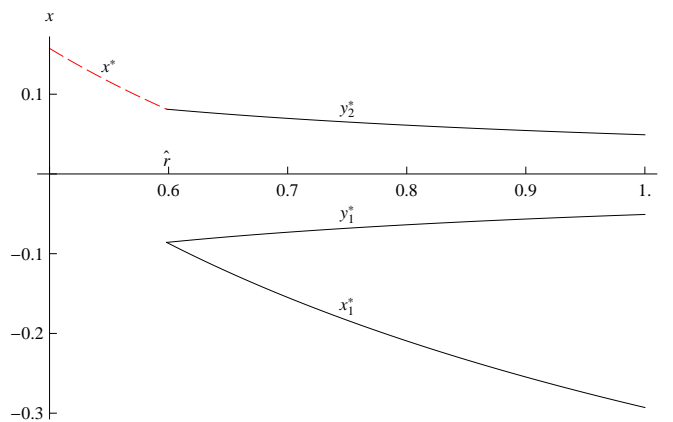}
\caption{\small Optimal Stopping Boundaries; with $K=1$  and $\beta=0.55$.}\label{2reunaongelmab}
\end{center}
\end{figure}

\section{Conclusions}

We studied a class of optimal stopping problems for SBM. We showed that the local directional predictability resulting from the presence of a skew point has a nontrivial and somewhat surprising impact on the optimal stopping policy of the underlying diffusion. More precisely, we delineated a set of relatively weak monotonicity conditions satisfied by a large class of exercise payoffs under which the skew point is always included in the continuation region. In that case postponing rational exercise is always worthwhile on a neighborhood of the skew point. An interesting implication of this finding is that the problem can become a three-boundary problem even when the exercise payoff is linear. We also analyzed the comparative static properties of the value and optimal timing policy and established that the value is an increasing function of skewness for increasing payoffs. In accordance with this observation higher skewness expands the continuation region and in that way increases the incentives to wait.

There are two natural directions towards which our analysis could be extended. First,  given that skewness can be introduced also for other diffusions beyond Brownian motion, it would be naturally of interest to consider how the singularity of the underlying diffusion affects the optimal stopping strategies and their values within a more general modeling framework. Second, given the close connection of optimal stopping with impulse control and bounded variation control problems, it would naturally be of interest to investigate how skewness affects the optimal policies in those associated problems. Both these extensions are out of the scope of this study and left for future research.\\

\noindent{\bf Acknowledgement:} The authors are grateful to {\em Sören Christensen} for constructive comments.

\end{document}